\documentclass[12pt, reqno]{amsart}
\usepackage{amssymb, amsmath, amsthm, cancel, hyperref, color, enumerate, verbatim, mathtools}

\usepackage[margin=1in, includeheadfoot]{geometry}

\newtheorem{theorem}{Theorem}

\newtheorem{example}[theorem]{Example}
\newtheorem{corollary}[theorem]{Corollary}
\newtheorem{proposition}[theorem]{Proposition}

\newcommand{\leg}[2]{\genfrac{(}{)}{}{}{#1}{#2}} 
\def\Z{{\mathbb Z}}

\begin{document}

\title[Incongruences]{Incongruences for modular forms\\ and applications to partition functions}
\author{Sharon Garthwaite and Marie Jameson}

\maketitle

\section*{Abstract}

The study of arithmetic properties of coefficients of modular forms $f(\tau) = \sum a(n)q^n$ has a rich history, including deep results regarding congruences in arithmetic progressions.  Recently, work of C.-S. Radu, S. Ahlgren, B. Kim, N. Andersen, and S. L\"{o}brich have employed the $q$-expansion principle of P. Deligne and M. Rapoport in order to determine more about where these congruences can occur.  Here, we extend the method to give additional results for a large class of modular forms.  We also give analogous results for generalized Frobenius partitions and the two mock theta functions $f(q)$ and $\omega(q).$


\section{Introduction and statement of results}\label{intro}

Let $p(n)$ be the number of partitions of $n$, that is, the number of ways to write $n$ as a sum of non-increasing sequence of positive integers.  (See, for example, \cite{Andrews98}).  The well-known Ramanujan congruences for the partition function state that for all $n$, we have
\begin{align*}
p(5n+4) &\equiv 0\pmod{5}\\
p(7n+5) &\equiv 0\pmod{7}\\
p(11n+6) &\equiv 0\pmod{11}.
\end{align*}
These congruences are so surprising and beautiful that one is left to wonder if there any other congruences of this kind.  Thus we define a Ramanujan congruence to be a congruence of the form $p(\ell n+t) \equiv 0\pmod{\ell},$ where $\ell$ is prime; S. Ramanujan conjectured that the only Ramanujan congruences for the partition function are the three given above. I. Kiming and J. Olsson \cite{KimingOlsson92} proved that for primes $\ell\geq 5,$ a Ramanujan congruence for the partition function must satisfy $t \equiv (1-\ell^2)/24 \pmod{\ell}.$ Later, S. Ahlgren and M. Boylan \cite{AhlgrenBoylan03} built on this result to prove that Ramanujan's conjecture is true. 

This work has inspired a number of results that describe where congruences can (or, rather, cannot) occur for Fourier coefficients of modular forms and mock theta functions \cite{Dewar11, Andersen14, AhlgrenKim15, Choi16, Loebrich17}. In this work, we give additional results regarding the non-existence of congruences.

Let $B \in \Z$, $k \in \frac{1}{2}\Z$, and $N\in \Z^+$.  As in \cite{AhlgrenKim15, Andersen14}, we set
\[\mathcal{S}(B,k,N,\chi) := \{\eta^B(\tau)F(\tau) : F(\tau) \in M_k^!(\Gamma_0(N),\chi)\}\]
where $M_k^!(\Gamma_0(N),\chi)$ is the space of weakly holomorphic modular forms of weight $k$ and level $N$ with character $\chi.$ (See Section \ref{prelim} for additional information.) In particular, if $f(\tau) \in S(B,k,N,\chi),$ then we have
\[f(\tau) = q^{B/24}\sum_{n\geq n_0}a(n)q^n, \qquad q:=e^{2\pi i\tau},\]
and any poles are supported on the cusps.

The following theorem can be obtained by applying the $q$-expansion principle of P. Deligne and M. Rapoport and is very much inspired by the work of Ahlgren, B. Kim, N. Andersen, and S. L\"obrich \cite{AhlgrenKim15, Andersen14, Loebrich17}.

\begin{theorem}\label{mainthm}
Let $\ell$ be prime, let $f(\tau) = q^{B/24}\sum_{n\geq n_0}a(n)q^n \in \mathcal{S}(B,k,N,\chi)$ have rational $\ell$-integral coefficients, and let $m\in \mathbb{Z}^+.$ Let $t_0 \in \mathbb{Z}$ such that
\[\ell \nmid a(t_0) \qquad \text{and}\qquad a(n)= 0 \text{ for all } n<t_0 \text{ with } n\equiv t_0 \pmod{m}.\]
Then for all $t\in \{0, \dots, m-1\}$ such that
\[t \equiv t_0d^2+B\frac{d^2-1}{24}\pmod{m}\]
for some integer $d$ with $\gcd(d,6Nm)=1,$ we have that
\[\sum a(mn+t)q^n\not\equiv 0\pmod{\ell}.\]
\end{theorem}

\begin{example}
Consider $f(\tau) = q^{-1/24}\sum_{n=0}^\infty p(n)q^n$ (so $B=-1,$ $k=0,$ $N=1,$ and $\chi$ is trivial). For $m=\ell=5,$ Theorem \ref{mainthm} guarantees that since $p(0) = 1,$ we have for $t = 0,3$ that 
\[\sum p(5n+t)q^n \not\equiv 0\pmod{5}.\] Similarly, since $p(1) = 1,$ we have for $t = 1,2$ that \[\sum p(5n+t)q^n \not\equiv 0\pmod{5}.\] This shows that that the only possible Ramanujan congruence for the partition function modulo 5 is the famous congruence
\[\sum p(5n+4)q^n \equiv 0\pmod{5}.\]
\end{example}

As noted by L\"obrich in \cite{Loebrich17}, a calculation similar to the above example can be executed for any prime $m=\ell\geq 5$ with $\leg{-23}{\ell} = -1$ to show that a Ramanujan congruence $\sum p(\ell n+t)q^n \equiv 0\pmod{\ell},$ if it exists, must satisfy $t \equiv (1-\ell^2)/24.$ This (partially) recovers the well-known theorem of Kiming and Olsson \cite{KimingOlsson92} mentioned above, which proves the statement for any prime $\ell \geq 5$.  Here, we generalize this approach to give a result that applies to any $f(\tau) \in \mathcal{S}(B,k,N,\chi).$

\begin{corollary} \label{cortomainthm}
Let $\ell$ be prime, let $f(\tau)$ be as in Theorem \ref{mainthm}, and let $m \in \mathbb{Z}^+.$  Let $n_1$ be the smallest integer such that $a(mn_1)\neq 0$ (if it exists).  If $\ell \nmid a(mn_1)$ then for each $d$ coprime to $6Nm$ and
\[t \equiv \frac{Bd^2-B}{24} \pmod{m},\]
we have that
\[\sum a(mn+t)q^n\not\equiv 0\pmod{\ell}.\]
Similarly, let $n_2$ be the smallest integer such that $a(mn_2+1)\neq 0$ (if it exists). If $\ell \nmid a(mn_2+1)$ then for each $d$ coprime to $6Nm$ and
\[t \equiv \frac{(B+24)d^2-B}{24}  \pmod{m},\]
we have that
\[\sum a(mn+t)q^n\not\equiv 0\pmod{\ell}.\]

In particular, for primes $\ell \geq 5$ and $m=\ell$, if $\ell \nmid a(\ell n_1)$, $\ell \nmid a(\ell n_2+1)$, and $\leg{B(B+24)}{\ell}=-1$ then $\sum a(\ell n+t)q^n\not\equiv 0\pmod{\ell}$ for all $t$ except possibly $t \equiv B(\ell^2-1)/24 \pmod{\ell}$.
\end{corollary}

\begin{proof}[Proof of Corollary \ref{cortomainthm}]
The first part of the claim follows directly from Theorem \ref{mainthm} with $t_0 = mn_1 \equiv 0 \pmod{m}$, and the second part of the claim follows similarly with $t_0 = mn_2+1 \equiv 1 \pmod{m}.$

For the last part, we observe, as in \cite{Loebrich17}, that $d$ will run through all nonzero congruence classes modulo $\ell$, so $d^2$ will take on $(\ell-1)/2$ distinct values modulo $\ell.$  Additionally, we see that $t \equiv t_0d^2+B \frac{d^2-1}{24} \pmod{\ell}$ is equivalent to 
\[
\leg{24t+B}{\ell} = \leg{24t_0 +B}{\ell}.
\]
Thus, if $\leg{B(B+24)}{\ell}=-1$ then $(Bd_1^2-B)/24 \not\equiv ((B+24)d_2^2-B)/24 \pmod{\ell}$ for any $d_1, d_2$.  Thus the first two parts of the corollary guarantee that $\sum a(\ell n+t)q^n\not\equiv 0\pmod{\ell}$ for $\ell-1$ incongruent values of $t$.  The remaining value corresponds to $d\equiv 0 \pmod{\ell}$.
\end{proof}

These results apply to many other partition-theoretic functions whose generating functions are modular forms, such as generalized $k$-colored Frobenius partitions (which are defined in Section \ref{gfp}).  We can also extend them to the well-known classical mock theta functions $f(q)$ and $\omega(q)$, which have partition-theoretic interpretations, but are not quite modular.

This paper is organized as follows.  In Section \ref{prelim}, we state the definitions and results needed to prove Theorem \ref{mainthm}. In Section \ref{proof} we prove our main theorem, Theorem \ref{mainthm}. In Section \ref{gfp}, we apply our result to the generalized Frobenius partitions, giving various corollaries and a new congruence for $c\phi_5(n).$ In Section \ref{mock}, we apply our results to mock theta functions $f(q)$ and $\omega(q)$.  Finally, in Section \ref{conc}, we discuss the possibilities for further results for other mock theta functions, including $\nu(q).$


\section{Preliminaries}\label{prelim}

First we recall some key definitions; see \cite{Ono04} for details. For $k$ an integer or half integer, $N$ a positive integer (where $4\mid N$ if $k \notin \mathbb{Z}$), and $\chi$ a Dirichlet character modulo $N$, we let $M_k^!(\Gamma_0(N),\chi)$ be the space of weakly holomorphic modular forms of weight $k$ with Nebentypus $\chi$ on $\Gamma_0(N).$

As in Section \ref{intro}, we set
\[\mathcal{S}(B,k,N,\chi) := \{\eta^B(\tau)F(\tau) : F(\tau) \in M_k^!(\Gamma_0(N),\chi)\},\]
where
\[\eta(\tau) := q^{1/24}\prod_{n=1}^\infty (1-q^n), \qquad q:=e^{2\pi i\tau}.\]
If $f(\tau) \in S(B,k,N,\chi),$ then we have
\[f(\tau) = q^{B/24}\sum_{n\geq n_0}a(n)q^n.\]
For any positive integer $m$ and any integer $t$ we define $\zeta_m := \exp(2\pi i/m)$ and
\begin{align*}
f_{m,t}(\tau) &:= \frac{1}{m}\sum_{\lambda = 0}^{m-1}\zeta_m^{-\lambda(t+B/24)}f\left(\frac{\tau+\lambda}{m}\right)\\
&= q^{(t+B/24)/m}\sum a(mn+t)q^n
\end{align*}
as in \cite[Section~5]{AhlgrenKim15}. 

In order to state a transformation law for $f_{m,t}$, set
\begin{equation} \label{N_m}
N_m := m, 8m, 3m, \text{or } 24m,
\end{equation}
according to whether $\gcd(m,6)=1, 2, 3, \text{or } 6$, respectively.  Define the slash operator as follows: for an analytic function $f$ on the upper half plane, an integer $k$, and a matrix $\gamma =(\begin{smallmatrix}a&b\\c&d\end{smallmatrix}) \in \mathrm{SL}_2(\mathbb{Z}),$ we have that
\[(f |_k\gamma)(\tau) := (c\tau + d)^{-k}f\left(\frac{a\tau + b}{c\tau + d}\right).\]
Work of Ahlgren and Kim gives the following transformation property.

\begin{proposition}[{\cite[pages 126-127]{AhlgrenKim15}}] \label{prop52}
Suppose that $f$ is as above, and define $\kappa := 24mN(k+B/2).$ For every $A = (\begin{smallmatrix}a&b\\ c&d\end{smallmatrix}) \in \Gamma_0(NN_m)$ with $\gcd(a,6)=1,$ we have
\[f_{m,t}^{24mN}|_\kappa A = f_{m,t_A}^{24mN},\]
where
\begin{equation}
\label{t_A}
t_A \equiv ta^2+B\frac{a^2-1}{24}\pmod{m}.
\end{equation}
In particular, there exists $s\in \mathbb{N}$ such that \[\Delta^sf_{m,t}^{24mN} \in M_{\kappa + 12s}(\Gamma_1(NN_m)).\]
\end{proposition}

Finally, we require the following fact, called the $q$-expansion principle, which was originally due to Deligne and Rapoport (although the precise statement we use here is a corollary due to C.-S. Radu).
\begin{theorem}[{\cite[Corollaire 3.12]{DeligneRapoport73}}, {\cite[Corollary 5.3]{Radu12}}]\label{qexpprinc}
Let $k, N$ be positive integers and let $f\in M_k (\Gamma_1(N)).$ If $f$ has coefficients in $\mathbb{Z}[\zeta_N]$ then the same is true of $f|_k \gamma$ for each $\gamma \in \Gamma_0(N).$ 
\end{theorem}


\section{Proof of Theorem \ref{mainthm}} \label{proof}

Assume without loss of generality that $f$ has integral coefficients and consider $t\in \{0, \ldots, m-1\}$ such that
\[t \equiv t_0d^2+B\frac{d^2-1}{24}\pmod{m}\]
for some integer $d$ with $\gcd(d,6Nm)=1.$ First note that we can find a matrix $A = (\begin{smallmatrix}a&b\\ c&d\end{smallmatrix}) \in \Gamma_0(NN_m)$ with $\gcd(a,6)=1,$ and $t_A \equiv t_0 \pmod{m},$ where $t_A$ is defined as in \eqref{t_A}. By Proposition \ref{prop52}, we then have

\begin{align*}
\left(f_{m,t}^{24mN}|_\kappa A \right)(\tau) &= f_{m,t_0}^{24mN}(\tau)\\
&= \left[q^{(t_0+B/24)/m}\sum_{n=0}^\infty a(mn+t_0)q^n \right]^{24mN}\\
&= a(t_0)^{24mN}q^{(24t_0+B)N}\left( 1 + O(q)\right).
\end{align*}

Now suppose for the sake of contradiction that $f_{m,t}\equiv 0\pmod{\ell}.$ Then  by Proposition \ref{prop52} we have that
\[\left(\ell^{-1}f_{m,t}\right)^{24mN}\Delta^s \in M_{\kappa+12s}(\Gamma_1(NN_m))\cap \mathbb{Z}[\zeta_{NN_m}][[q]].\]
so by the $q$-expansion principle (see Theorem \ref{qexpprinc}), we have that
\[\left(\ell^{-1}f_{m,t}\right)^{24mN}\Delta^s|_{\kappa+12s} A \in  \mathbb{Z}[\zeta_{NN_m}][[q]].\]
But by the above argument, this is
\[\left(\ell^{-24mN}f_{m,t}^{24mN}\right)\Delta^s|_{\kappa+12s} A = \left(\frac{1}{\ell}a(t_0)\right)^{24mN}q^{(24t_0+B)N + s}\left( 1 + O(q)\right),\]
which only holds if $\ell \mid a(t_0),$ a contradiction. \qed


\section{Generalized Frobenius partitions} \label{gfp}

In 1984, G. Andrews defined combinatorial objects called $k$-colored generalized Frobenius partitions.  Given a positive integer $k,$ the generating function for their counting function $c\phi_{k}(n)$ is given by \cite[Theorem 5.2]{Andrews84}
\[\sum_{n=0}^{\infty} c\phi_{k}(n)q^n = \frac{\displaystyle\sum_{m_1, m_2, \dots, m_{k-1}\in \mathbb{Z}}q^{Q(m_1, m_2, \dots, m_{k-1})}}{\displaystyle\prod_{n=1}^{\infty}(1-q^n)^k},\]
where $Q(m_1, m_2, \dots, m_{k-1})$ is
\[Q(m_1, m_2, \dots, m_{k-1}) := \sum_{i=1}^{k-1}m_i^2+\sum_{1\leq i<j\leq k-1}m_im_j.\] One can use the theory of theta functions to relate this generating function to a modular form (see, for example, \cite[Lemma 1]{JamesonWieczorek18}); for some character $\chi_k,$ we have that
\begin{equation}\label{thetafcn}
\sum_{m_1, m_2, \dots, m_{k-1}\in \mathbb{Z}}q^{Q(m_1, m_2, \dots, m_{k-1})} \in M_{\frac{k-1}{2}}^!(\Gamma_0(N),\chi_k),
\end{equation}
where $N=k$ or $2k$ according to whether $k$ is odd or even.

Generalized Frobenius partitions are of interest here because they are known to exhibit congruences analogous to those satisfied by $p(n)$ and their generating functions are related to the theory of modular forms, but they are not expected to be generalized eta-quotients for $k\geq 3$ (so the work of \cite{Loebrich17} does not apply). Thus we set
\[f_k(\tau) := q^{-k/24}\sum_{n=0}^{\infty} c\phi_{k}(n)q^n \in \mathcal{S}(-k,(k-1)/2,N,\chi_k)\]
and apply Theorem \ref{mainthm} to give the following statement.

\begin{corollary} \label{cphikthm}
For $\ell$ prime and $m\in \mathbb{Z}^+,$ if $t_0 \in \{0, \dots, m-1\}$ such that $\ell \nmid c\phi_k(t_0)$ then for all $t\in \{0, \dots, m-1\}$ such that
\[t \equiv t_0d^2-k\frac{d^2-1}{24}\pmod{m}\]
for some integer $d$ with $\gcd(d,6mk)=1,$ we have that
\[\sum c\phi_k(mn+t)q^n\not\equiv 0\pmod{\ell}.\]
\end{corollary}

Let us pause for a moment to consider a very specific example which illustrates how this corollary compares to similar results from the literature.
\begin{example}
We now determine which values of $t$ allow us to conclude
\begin{equation}\label{example}\sum c\phi_3(10n+t)q^n \not\equiv 0\pmod{5}.\end{equation}
A result of Ahlgren and Kim \cite[Theorem 1.3]{AhlgrenKim15} does apply to 3-colored generalized Frobenius partitions but not to congruences modulo 5. Results of L\"obrich \cite{Loebrich17} do not apply here since the generating function for $c\phi_3(n)$ is not a generalized eta-quotient.  From work of Andersen \cite[Theorem 5.1]{Andersen14}, it follows that \eqref{example} holds for $t$ such that $\leg{1-8t}{5}=1,$ i.e., 
\[t \equiv 0, 4, 5, \text{or }9 \pmod{10}.\]
From Corollary \ref{cphikthm} (with $t_0 = 0, 1$), it follows that \eqref{example} holds for
\[t \equiv 0, 1, 3, 4, 5, 6, 8, \text{or }9 \pmod{10}.\]
Thus, we have shown that \eqref{example} must hold for all $t$ except $t \equiv 2$ or $7 \pmod{10}$  (and some quick calculations can also verify that there are no congruences at all of the form $\sum c\phi_3(10n+t)q^n \equiv 0\pmod{5}$).
\end{example}

We now study cases for which we can prove incongruences for all but one or two congruence classes modulo $m$.

\subsection{Congruences of the form $\sum c\phi_k(\ell n+t)q^n \equiv 0 \pmod{\ell}$} First, we consider Ramanujan congruences for generalized Frobenius partitions, i.e., congruences where $m=\ell.$  When $k=1,$ we have $c\phi_1(n) = p(n)$, and so we have only the Ramanujan congruences for the partition function given above.  For $k=2,$ M. Dewar proved that the only Ramanujan congruences are
\begin{align*}
c\phi_2(2n+1) \equiv 0\pmod{2}\\
c\phi_2(5n+3) \equiv 0\pmod{5},
\end{align*}
and several other examples of Ramanujan congruences for $k\geq 3$ can be found in the literature.  In general, work of D. Choi implies for a fixed $k$ there are only finitely many Ramanujan congruences \cite[Section 1.1.2]{Choi16}.

Since one can show that $c\phi_k(0) = 1$ and $c\phi_k(1) = k^2,$ Corollary \ref{cortomainthm} gives the following analogue to the theorem of Kiming and Olsson \cite{KimingOlsson92} for primes $\ell$ relatively prime to $6k$ such that $\leg{k(k-24)}{\ell}=-1$. 

\begin{corollary}\label{Cor8}
For primes $\ell \geq 5$ such that $\ell \nmid k,$ if $\leg{k(k-24)}{\ell}=-1$ then
\[\sum c\phi_k(\ell n+t)q^n\not\equiv 0\pmod{\ell}\]
for all $t$ except possibly $t \equiv k(1-\ell^2)/24 \pmod{\ell}$.
\end{corollary}

For example, if $k=3$, then we see that Corollary \ref{Cor8} applies to all primes $\ell \equiv 3,5,6 \pmod{7}$.  Additionally, we note that this is consistent with the beautiful congruences of F. Garvan and J. Sellers \cite{GS} which state that for any non-negative integer $N$
\begin{align*}
c\phi_{5N+1}(5n+4) &\equiv 0\pmod{5}\\
c\phi_{7N+1}(7n+5) &\equiv 0\pmod{7}\\
c\phi_{11N+1}(11n+6) &\equiv 0\pmod{11}
\end{align*}
and in fact, Corollary \ref{Cor8} shows that these are the only Ramanujan congruences of that type.

\subsection{Congruences of the form $\sum c\phi_k(2\ell n+t)q^n \equiv 0 \pmod{\ell}$}  Next, we consider congruences where $m=2\ell$ to see how to extend the approach in the proof of Corollary \ref{cortomainthm}.  When $k$ is odd, we find that our proof is very similar to what happens in the $m=\ell$ case while $k$ even requires more information.  

Suppose $k$ is odd and suppose $\ell \nmid k$.  We apply Corollary \ref{cphikthm} with $t_0 = 0,1$.  Here we note that if $d$ is such that $\gcd(d,6\ell)=1$, then 
\begin{align*}
t_0d^2-k\frac{d^2-1}{24} \equiv t_0(d+12k\ell)^2-k\frac{(d+12k\ell)^2-1}{24} \pmod \ell;\\
t_0d^2-k\frac{d^2-1}{24} \not\equiv t_0(d+12k\ell)^2-k\frac{(d+12k\ell)^2-1}{24} \pmod{2}.
\end{align*} 
Hence, for each $t_0$, as we allow $d$ to run through each of the $\phi(2\ell) = \ell -1$ congruence classes modulo $2\ell$ (where $d$ is relatively prime to $2\ell,$ and we choose a representative $d$ so that $\gcd(d,3k)=1$), we obtain $(\ell-1)/2$ incongruent values of $t$.  For each $d$, we also consider $d+12k\ell$, obtaining another $(\ell-1)/2$ incongruent values of $t$ for a total of $\ell-1$ incongruent values of $t$ for each $t_0$. If $\leg{k(k-24)}{\ell} = -1$, then there is no overlap in the $t$ values we find using $t_0=0,1$ as
\[
\leg{24t-k}{\ell} = \leg{24t_0 -k}{\ell}.
\]  
This gives us a total of $2\ell-2$ values of $t$ that yield incongruences.  The two remaining $t$ values are those for which $\leg{24t-k}{\ell} = 0$.

Now suppose that $k$ is even, $\ell \geq 5$ is prime, and that $\ell$ does not divide any of the following:
\begin{align*}
c\phi_k(0) &= 1,\\
c\phi_k(1) &= k^2,\\
c\phi_k(2) &= 2k^2 + \binom{k}{2}^2 = \frac{k^2(k^2-2k+9)}{4},\\
c\phi_k(3) &= 3k^2 + 2k^2\binom{k}{2} + \binom{k}{3}^2 = \frac{k^2(k^4-6k^3+49k^2-48k+112)}{36}.
\end{align*}
For example, if $k=2$ then we consider all primes $\ell \geq 5,$ if $k=4$ then we consider all primes $\ell \geq 5$ with $\ell \neq 17$, and if $k=6$ then we consider all primes $\ell \geq 5$ with $\ell \neq 11,$ $\ell \neq 397$. 

We now apply Corollary \ref{cphikthm} with $t_0 = 0,2.$ As $d$ runs through each of the $\phi(2\ell) = \ell -1$ congruence classes modulo $2\ell$ (where $d$ is relatively prime to $2\ell,$ and we choose a representative $d$ so that $\gcd(d,3k)=1$), we obtain
\[t \equiv t_0d^2-k\frac{d^2-1}{24}\pmod{2\ell},\]
where $t\equiv t_0 \pmod{2}$ since $t_0$ and $k$ are even. If $\leg{k(k-48)}{\ell} = -1$ then it follows that $\sum c\phi_k(2\ell n+t)q^n \not\equiv 0\pmod{\ell}$ for $\ell-1$ incongruent values of $t$ (which are all even).
Similarly, we apply Corollary \ref{cphikthm} with $t_0 = 1,3$ and find that if $\leg{(24-k)(72-k)}{\ell} = -1$ then we have that $\sum c\phi_k(2\ell n+t)q^n \not\equiv 0\pmod{\ell}$ for $\ell-1$ incongruent values of $t$ (which are all odd). 
The two remaining $t$ values are those for which $\leg{24t-k}{\ell} = 0$.

This proves the following statement.

\begin{corollary}
For odd $k$ and primes $\ell \geq 5$ such that $\ell \nmid k$, if $\leg{k(k-24)}{\ell} = -1$ then
\[\sum c\phi_k(2\ell n+t)q^n\not\equiv 0\pmod{\ell}\]
for all $t$ except possibly $t \equiv k(1-\ell^2)/24$ or $t \equiv k(1-\ell^2)/24 + \ell \pmod{2\ell}$. 

For even $k$ and primes $\ell \geq 5$ such that $\ell \nmid k(k^2-2k+9)(k^4-6k^3+49k^2-48k+112),$ if $\leg{k(k-48)}{\ell} = \leg{(24-k)(72-k)}{\ell} = -1$ then
\[\sum c\phi_k(2\ell n+t)q^n\not\equiv 0\pmod{\ell}\]
for all $t$ except possibly $t \equiv k(1-\ell^2)/24$ or $t \equiv k(1-\ell^2)/24 + \ell \pmod{2\ell}$. 
\end{corollary}

\subsection{Searching for new congruences} Corollary \ref{cphikthm} can be used to describe where congruences are prohibited; this allows us to search for congruences in the remaining arithmetic progressions more efficiently.  For example, numerical data for $c\phi_5(n)$ (using, for example, Theorem 3.1 of \cite{ChanWangYang19}) suggests the following newly discovered congruences.
\begin{theorem}
\[c\phi_5(325n+t) \equiv 0 \pmod{13} \text{ for }t = 15, 25, 50, 75, 90, 100, 115, 140, 165, 175, 240, 275.\]
\end{theorem}
We can prove these using standard facts from the theory of modular forms.

\begin{proof}
By equation \eqref{thetafcn} together with well-known results about eta-quotients (see, for example, Theorems 1.64 and 1.65 of \cite{Ono04}), we first have that
\[\sum_n c\phi_5(n)q^n \prod_n (1-q^n)^5 \frac{\eta^8(\tau)}{\eta(13\tau)} \eta^{17}(325\tau) \in M_{14}(\Gamma_0(325),\chi)\]
for some character $\chi.$ Thus we set
\[g(\tau) = \sum_n a(n)q^n := \left(\sum_n c\phi_5(n)q^n \prod_n (1-q^n)^5 \frac{\eta^8(\tau)}{\eta(13\tau)} \eta^{17}(325\tau) \right) \mid U(5) \in M_{14}(\Gamma_0(325),\chi)\]
(see, for example, Section 2.4 of \cite{Ono04}) and note that
\[g(\tau) \equiv \sum_n c\phi_5(5n-230)q^{n} \prod_n (1-q^{65n})^{17} \pmod{13}.\]
Now, by the theory of quadratic twists (see, for example, Section 2.2 of \cite{Ono04}), it follows that
\[\sum_{\leg{n}{5} = \leg{n}{13} = 1} a(n)q^n \in M_{14}(\Gamma_0(1373125),\chi).\]
By a theorem of Sturm, one can prove that
\[\sum_{\leg{n}{5} = \leg{n}{13} = 1} a(n)q^n \equiv 0\pmod{13}\]
by checking that the first 2070251 coefficients of this form vanish modulo 13. It follows that
\[\sum_{\leg{n}{5} = \leg{n}{13} = 1} c\phi_5(5n-230)q^{n} \equiv 0\pmod{13},\]
completing the proof.
\end{proof}


\section{Mock Theta Functions} \label{mock}
Let 
\[
f(q) = 1+\sum_{n=1}^\infty \frac{q^{n^2}}{(1+q)^2(1+q^2)^2\dots(1+q^n)^2} =:\sum_{n=0}^\infty a_f(n)q^n
\]
and
\[
\omega(q) = 1+\sum_{n=1}^\infty \frac{q^{2n^2+2n}}{(1+q)^2(1+q^3)^2\dots(1+q^{2n+1})^2} =:\sum_{n=0}^\infty a_{\omega}(n)q^n
\]
be two of Ramanujan's classical third order mock theta functions.  Both $a_f(n)$ and $a_{\omega}(n)$ have partition theoretic interpretations. (See for example \cite{Andrews07}.)  

In \cite{GarthwaitePenniston08} the first author and D. Penniston proved that such functions have an infinite number of non-nested congruences similar to the classical partition function (see \cite{Ono00}, \cite{AhlgrenOno01}) and, more generally, weakly holomorphic modular forms (see \cite{Treneer06}).  In particular \cite[Corollary 4.3]{GarthwaitePenniston08}, for any prime $p\geq 5$, any prime $\ell \nmid 6p$ and any $j\geq 1$, there exists a
positive integer $m$ such that a positive proportion of the primes $Q$ have the property that 
\[
a_\omega\left(\frac{Q^3\ell^mn-2}{3}\right) \equiv 0 \pmod{\ell^j}
\]
for all $n$ with $\gcd(n,\ell Q) =1$ and $\leg{-Q^3\ell^mn}{p}=-1$.  M. Waldherr \cite[Theorem 1.1]{Waldherr11} gave the first concrete example of a congruence for this function, proving
\[ 
a_\omega(40n+27) \equiv a_\omega(40n+35) \equiv 0 \pmod{5}.
\] 

Ahlgren and Kim \cite[Theorems 1.1 and 1.2]{AhlgrenKim15} proved that 
\[
\sum a_f(mn+t)q^n \not\equiv 0 \pmod{3}
\]
and 
\[
\sum a_\omega(mn+t)q^n \not\equiv 0 \pmod{3}
\]
for any positive integer $m$ and any integer $t$ by adapting Radu's work for the classical partition function \cite{Radu12} to these mock theta functions.  

Andersen \cite{Andersen14} built on the work of Ahlgren and Kim to prove the following results. 

\begin{theorem}[Theorems 1.1 and 1.3 of \cite{Andersen14}] Suppose  $\ell \geq 5$ is prime. If for all $n$ we have $a_f(mn+t) \equiv 0 \pmod{\ell}$ then $\ell \mid m$ and $\leg{24t-1}{\ell} \neq \leg{-1}{\ell}$.  Similarly, if $\ell \geq 5$ is prime and if for all $n$ we have $a_\omega(mn+t) \equiv 0 \pmod{\ell}$ then $\ell \mid m$ and $\leg{3t+2}{\ell} \neq \leg{-1}{\ell}$. 
\end{theorem}

Thus, for primes $\ell\geq 5$, it follows that about half of the possible $t$ values lead to incongruences.  We now prove additional incongruences.

\begin{theorem}\label{mock.cong} Suppose  $\ell \geq 5$ is prime.  Let $m\in \mathbb{Z}^+$ and $t_{0,f}\in \{0, \dots, m-1\}$.  If $\leg{1-24t_{0,f}}{p} = -1$ for some $p\mid m$ and $\ell \nmid a_f(t_{0,f})$
then for all $t \in \{0,\dots,m-1\}$ such that 
\[
t \equiv t_{0,f}d^2+\frac{1-d^2}{24}\pmod{m}.
\]
for some integer $d$ with $\gcd(d,6Nm)=1$, we have that
\[\sum a_f(mn+t)q^n\not\equiv 0\pmod{\ell}.\]

Similarly, let $t_{0,\omega}\in\{0,\dots, m-1\}$.  If $\leg{-3t_{0,\omega}-2}{p} = -1$ for some $p\mid m$ and $\ell \nmid a_{\omega}(t_{0,\omega})$
then for all $t \in \{0,\dots,m-1\}$ such that 
\[
t \equiv t_{0,\omega}d^2+\frac{2}{3}(d^2-1)\pmod{m}.
\]
for some integer $d$ with $\gcd(d,6Nm)=1$, we have that
\[\sum a_\omega(mn+t)q^n\not\equiv 0\pmod{\ell}.\]
\end{theorem}

These mock theta congruence and incongruence results all make use of the fact that $f(q)$ and $\omega(q)$ inherit their automorphic properties from the holomorphic parts of a vector-valued harmonic Mass form (see for example \cite{Zwegers01}).    To follow the method used in the proof of Theorem \ref{mainthm}, we want to consider pairs $(m,t)$ for which
\begin{equation}
\label{Hf}
H_{f,m,t} := q^{\frac{t-1/24}{m}}\sum a_f(mn+t)q^n
\end{equation}
and 
\begin{equation}
\label{Hom}
H_{\omega, m,t} := q^{\frac{t+2/3}{m}}\sum a_\omega(mn+t)q^n
\end{equation}
are weakly homolorphic modular forms.  For $H_{f,m,t}$ this happens whenever $mn-t \neq k(3k+1)/2$ for any $k,n \in \Z$, as the coefficients of the non-holomorphic part of the corresponding harmonic weak Maass form are zero on such progressions. It follows that $H_{f,m,t}$ is a weakly holomorphic modular form if $\leg{1-24t}{p}=-1$ for some prime $p\mid m.$  Similarly, for $H_{\omega,m,t}$ we want $mn-t-1 \neq 3k^2+2k$ for any $k,n \in \Z$, and it follows that $H_{\omega,m,t}$ is a weakly holomorphic modular form if $\leg{-3t-2}{p}=-1$ for some prime $p\mid m.$

\begin{proof}[Proof of Theorem \ref{mock.cong}]
Consider $t \in \{0,\dots, m-1\}$ and suppose there is a $t_{0,f} \in \{0, \dots, m-1\}$ and $d$ with $\gcd(d, 6Nm)=1$ for which 
\[
t \equiv t_{0,f}d^2+\frac{1-d^2}{24}\pmod{m}.
\]
Suppose there is a prime $p \mid m$ for which $\leg{24t_{0,f}-1}{p} = -1$, then 
\[
\leg{1-24t}{p} = \leg{1-24t_{0,f}}{p} = -1.
\]
 Let
\[N_{f,m} := 2m, 8m, 6m, \text{or } 24m\]
according to whether $\gcd(m,6)=1, 2, 3, \text{or } 6.$  Then, as in the proof of Theorem 1.1 of \cite{AhlgrenKim15},  $H_{f,m,t}$ as defined in \eqref{Hf}
satisfies the following transformation property \cite[Proposition 3.2]{AhlgrenKim15}: for every $A = (\begin{smallmatrix}a&b\\ c&d\end{smallmatrix}) \in \Gamma_0(N_{f,m})$ with $3 \nmid a$ we have
\[H_{f,m,t}^{24m} \mid_{12m} A = H_{f,m,t_{A,f}}^{24m}.\]
Then
\begin{align*}
(H_{f,m,t}^{24m} \mid_{12m} A)(\tau) &= H_{f,m,t_{A,f}}^{24m}(\tau)\\
&= a_f(t_{A,f})^{24m}q^{24t_{A,f}-1}(1 + O(q)).
\end{align*}
We now follow a similar argument as in the proof of Theorem \ref{mainthm} to prove the statement of the theorem for $f$. 

The statement of the theorem for $\omega$ follows in the same way.  As in the proof of Theorem 1.2 of \cite{AhlgrenKim15}, if we consider $t$ satisfying the hypotheses of Theorem \ref{mock.cong} then $H_{\omega, m,t}$ as defined in \eqref{Hom}
satisfies the following transformation property \cite[Proposition 4.1]{AhlgrenKim15}: for every $A = (\begin{smallmatrix}a&b\\ c&d\end{smallmatrix}) \in \Gamma_0(2N_{f,m})$ with $3 \nmid a$ we have
\[H_{\omega,m,t}^{24m}|_{12m} A = H_{\omega,m,t_{A,\omega}}^{24m}.\]
Then
\begin{align*}
(H_{\omega,m,t}^{24m}|_{12m} A)(\tau) &= H_{\omega,m,t_{A,\omega}}^{24m}(\tau)\\
&= a_\omega(t_{A,\omega})^{24m}q^{24t_{A,\omega}+16}(1 + O(q)),
\end{align*}
and we again apply the line of argument in the proof of Theorem \ref{mainthm} to complete the proof of Theorem \ref{mock.cong}.
\end{proof}

In particular, we can prove the following result.

\begin{corollary}\label{mock.Ramanujan}
Let $\ell \geq 5$ be prime.  Suppose $\leg{1-24i}{\ell} = -1$ for some $0 \leq i \leq 5$, then if $a_f(\ell n+ t) \equiv 0 \pmod{\ell}$, we have $t \equiv \frac{1-\ell^2}{24} \pmod{\ell}$.  

Similarly, suppose $\leg{-3i-2}{\ell} = -1$ for some $0 \leq i \leq 5$, then if $a_\omega(\ell n+ t) \equiv 0 \pmod{\ell}$, we have $t \equiv \frac{-2(1-\ell^2)}{3} \pmod{\ell}$.   
\end{corollary}

\begin{proof}
Fix a prime $\ell \geq 5$.  Now suppose $\sum a_f(\ell n+t)q^n \equiv 0 \pmod{\ell}$.  Andersen's result implies that $\leg{1-24t}{\ell} = -1$ or $\ell\mid (1-24t)$.  By hypothesis, there exists  some $0 \leq i \leq 5$ for which  $\leg{1-24i}{\ell} = -1$.  Take the smallest such $i$ and note that all prime factors of $a_f(i)$ are strictly less than $5$ for $0 \leq i\leq 5$. 

We now apply Theorem \ref{mock.cong} with $m=\ell $ and $t_{0,f}=i$.  As in the proof of Corollary \ref{cortomainthm}, we see $d$ will run through all nonzero congruence classes modulo $\ell$, so $d^2$ will take on $(\ell-1)/2$ distinct values modulo $\ell.$  Additionally, we see that
\[
s \equiv id^2+\frac{1-d^2}{24}\pmod{\ell} 
\]
implies  
\[
\leg{1-24s}{\ell} \equiv \leg{1-24i}{\ell} = -1.  
\]
These two facts together imply that $t$ will be one of the $(\ell-1)/2$ distinct values $s$ modulo $\ell$  for which $s \equiv id^2+\frac{1-d^2}{24}\pmod{\ell}$.

A similar argument holds for $\sum a_\omega(\ell n+t)q^n \equiv 0 \pmod{\ell}$, noting that all prime factors of $a_\omega(i)$ are strictly less than $5$ for $0 \leq i\leq 5$, and Andersen's theorem covers the case $\leg{-(3t+2)}{\ell} = 1$.
\end{proof}

\begin{example}
If we search for $t$ such that $\sum a_\omega(40n+t)q^n \equiv 0 \pmod{5}$ as in Waldherr's results, we see that $5\nmid a_\omega(t)$ for $0\leq t \leq 39$ except $t=6,20,23,24,27,35$. Andersen's result implies incongruences for $t=23, 24$. 
We see that $5\nmid a(12)$ together with Corollary \ref{mock.Ramanujan} implies an incongruence for $t=20$.  This leaves only $t=6,23$, and $37$. We find $a_\omega(40+6) \not\equiv 0 \pmod{5}$.  Waldherr proves congruences for $t= 23,37$.
\end{example}

\section{Concluding Remarks} \label{conc}
Theorem \ref{mock.cong} proves incongruences for the mock theta functions $f(q)$ and $\omega(q)$.  Similar incongruence results may extend to other mock theta functions as well.  For example, in \cite{AndrewsDixitYee15} Andrews, A. Dixit, and A. Yee provide a partition-theoretic interpretation to the third order mock theta function
\[
\nu(q):=\sum_{n=0}^\infty \frac{q^{n^2+n}}{(-q;q^2)_{n+1}} = q\omega(q^2)+(-q^2;q^2)^3_\infty(q^2;q^2)_\infty
\]
and prove 
\[
a_\nu(10n+8) \equiv 0 \pmod{5}
\]
for all $n \in \Z$ where $\nu(q)=:\sum a_\nu(n)q^n$.  

As a result of Corollary \ref{mock.Ramanujan}, we can say the following.

\begin{corollary}
If $\ell$ is prime and $\ell \equiv 5,7 \pmod{8}$ and $\nu(\ell n+t) \equiv 0 \pmod{\ell}$ for all $n \in \Z$ then $t \equiv \frac{\ell^2-1}{3} \pmod{\ell}$.
\end{corollary}

\begin{proof}
Suppose $\nu(\ell n+t) \equiv 0 \pmod{\ell}$  for all $n \in \Z$  where $\ell$ is prime.  We separate the progression $\ell n+t$ into subprogressions $2\ell n+t$ and $2\ell n+(\ell+t)$ and note that the odd coefficients of $\nu(q)$ come from $q\omega(q^2)$.  We see that if $t$ is odd then we must have $a_\omega(\ell n+(t-1)/2) \equiv 0 \pmod{\ell}$ for all $n\in \Z$, and if $t$ is even then we must have $a_\omega(\ell n+(\ell+t-1)/2) \equiv 0 \pmod{\ell}$ for all $n\in \Z$.  

If $\ell \equiv 5,7 \pmod{8}$, then $\leg{-2}{\ell} = -1$, so by Corollary \ref{mock.Ramanujan} we must have $(t-1)/2 \equiv -2(1-\ell^2)/3$ if $t$ is odd or $(\ell+t-1)/2 \equiv -2(1-\ell^2)/3$ if $t$ is even.
\end{proof}

Moreover, we note that the proof of Theorem \ref{mock.cong} requires the condition $\leg{1-24t_{0,f}}{p} = -1$ for some $p\mid m$ for $f(q)$ incongruences; however, we have not identified any results in the literature for which is a necessary condition, nor have we for the equivalent $\omega(q)$ condition.  It may be possible to remove this requirement.


\bibliographystyle{alpha}
\bibliography{refs}

\begin{thebibliography}{CWY19}

\bibitem[AB03]{AhlgrenBoylan03}
Scott Ahlgren and Matthew Boylan.
\newblock Arithmetic properties of the partition function.
\newblock {\em Invent. Math.}, 153(3):487--502, 2003.

\bibitem[ADY15]{AndrewsDixitYee15}
George~E. Andrews, Atul Dixit, and Ae~Ja Yee.
\newblock Partitions associated with the {R}amanujan/{W}atson mock theta
  functions {$\omega(q)$}, {$\nu(q)$} and {$\phi(q)$}.
\newblock {\em Res. Number Theory}, 1:Art. 19, 25, 2015.

\bibitem[AK15]{AhlgrenKim15}
Scott Ahlgren and Byungchan Kim.
\newblock Mock theta functions and weakly holomorphic modular forms modulo 2
  and 3.
\newblock {\em Math. Proc. Cambridge Philos. Soc.}, 158(1):111--129, 2015.

\bibitem[And84]{Andrews84}
George~E. Andrews.
\newblock Generalized {F}robenius partitions.
\newblock {\em Mem. Amer. Math. Soc.}, 49(301):iv+44, 1984.

\bibitem[And98]{Andrews98}
George~E. Andrews.
\newblock {\em The theory of partitions}.
\newblock Cambridge Mathematical Library. Cambridge University Press,
  Cambridge, 1998.
\newblock Reprint of the 1976 original.

\bibitem[And07]{Andrews07}
George~E. Andrews.
\newblock Partitions, {D}urfee symbols, and the {A}tkin-{G}arvan moments of
  ranks.
\newblock {\em Invent. Math.}, 169(1):37--73, 2007.

\bibitem[And14]{Andersen14}
Nickolas Andersen.
\newblock Classification of congruences for mock theta functions and weakly
  holomorphic modular forms.
\newblock {\em Q. J. Math.}, 65(3):781--805, 2014.

\bibitem[AO01]{AhlgrenOno01}
Scott Ahlgren and Ken Ono.
\newblock Congruence properties for the partition function.
\newblock {\em Proc. Natl. Acad. Sci. USA}, 98(23):12882--12884, 2001.

\bibitem[Cho16]{Choi16}
Dohoon Choi.
\newblock Congruences involving arithmetic progressions for weakly holomorphic
  modular forms.
\newblock {\em Adv. Math.}, 294:489--516, 2016.

\bibitem[CWY19]{ChanWangYang19}
Heng~Huat Chan, Liuquan Wang, and Yifan Yang.
\newblock Modular forms and {$k$}-colored generalized {F}robenius partitions.
\newblock {\em Trans. Amer. Math. Soc.}, 371(3):2159--2205, 2019.

\bibitem[Dew11]{Dewar11}
Michael Dewar.
\newblock Non-existence of {R}amanujan congruences in modular forms of level
  four.
\newblock {\em Canad. J. Math.}, 63(6):1284--1306, 2011.

\bibitem[DR73]{DeligneRapoport73}
P.~Deligne and M.~Rapoport.
\newblock Les sch\'{e}mas de modules de courbes elliptiques.
\newblock In {\em Modular functions of one variable, {II} ({P}roc. {I}nternat.
  {S}ummer {S}chool, {U}niv. {A}ntwerp, {A}ntwerp, 1972)}, pages 143--316.
  Lecture Notes in Math., Vol. 349, 1973.

\bibitem[GP08]{GarthwaitePenniston08}
Sharon~Anne Garthwaite and David Penniston.
\newblock {$p$}-adic properties of {M}aass forms arising from theta series.
\newblock {\em Math. Res. Lett.}, 15(3):459--470, 2008.

\bibitem[GS14]{GS}
Frank~G. Garvan and James~A. Sellers.
\newblock Congruences for generalized {F}robenius partitions with an
  arbitrarily large number of colors.
\newblock {\em Integers}, 14:Paper No. A7, 5, 2014.

\bibitem[JW18]{JamesonWieczorek18}
Marie Jameson and Maggie Wieczorek.
\newblock Congruences for modular forms and generalized frobenius partitions,
  2018.

\bibitem[KO92]{KimingOlsson92}
Ian Kiming and J\o rn~B. Olsson.
\newblock Congruences like {R}amanujan's for powers of the partition function.
\newblock {\em Arch. Math. (Basel)}, 59(4):348--360, 1992.

\bibitem[L\"17]{Loebrich17}
Steffen L\"{o}brich.
\newblock Linear incongruences for generalized eta-quotients.
\newblock {\em Res. Number Theory}, 3:Art. 18, 8, 2017.

\bibitem[Ono00]{Ono00}
Ken Ono.
\newblock Distribution of the partition function modulo {$m$}.
\newblock {\em Ann. of Math. (2)}, 151(1):293--307, 2000.

\bibitem[Ono04]{Ono04}
Ken Ono.
\newblock {\em The web of modularity: arithmetic of the coefficients of modular
  forms and {$q$}-series}, volume 102 of {\em CBMS Regional Conference Series
  in Mathematics}.
\newblock Published for the Conference Board of the Mathematical Sciences,
  Washington, DC; by the American Mathematical Society, Providence, RI, 2004.

\bibitem[Rad12]{Radu12}
Cristian-Silviu Radu.
\newblock A proof of {S}ubbarao's conjecture.
\newblock {\em J. Reine Angew. Math.}, 672:161--175, 2012.

\bibitem[Tre06]{Treneer06}
Stephanie Treneer.
\newblock Congruences for the coefficients of weakly holomorphic modular forms.
\newblock {\em Proc. London Math. Soc. (3)}, 93(2):304--324, 2006.

\bibitem[Wal11]{Waldherr11}
Matthias Waldherr.
\newblock On certain explicit congruences for mock theta functions.
\newblock {\em Proc. Amer. Math. Soc.}, 139(3):865--879, 2011.

\bibitem[Zwe01]{Zwegers01}
S.~P. Zwegers.
\newblock Mock {$\theta$}-functions and real analytic modular forms.
\newblock In {\em {$q$}-series with applications to combinatorics, number
  theory, and physics ({U}rbana, {IL}, 2000)}, volume 291 of {\em Contemp.
  Math.}, pages 269--277. Amer. Math. Soc., Providence, RI, 2001.

\end{thebibliography}

\end{document}